\documentclass{amsart}
\usepackage{amsmath,amssymb,verbatim}

\newtheorem{thm}{Theorem}

\newtheorem{lem}{Lemma}
\newtheorem{cor}{Corollary}
\theoremstyle{definition}

\newtheorem{remark}{Remark}

\newcommand{\D}{\mathbb D}
\newcommand{\N}{\mathbb N}
\newcommand{\Z}{\mathbb Z}
\newcommand{\Un}{1\hskip -1.25mm 1}
\begin{document}
\title[Quasihomogeneous Toeplitz Operators]{Quasihomogeneous Toeplitz Operators on the Harmonic Bergman Space}

\date{\today}

\author{Issam Louhichi} 
\address{King Fahd University of Petroleum \& Minerals\\
 Department of Mathematics \& Statistics\\
 Dhahran 31261, Saudi Arabia}
\email{issam@kfupm.edu.sa}

\author{Lova Zakariasy}
\thanks{The second author was partially supported by \emph{Agence Universitaire de la Francophonie}}
\address{High Institute of Technology\\201 Antsiranana\\Madagascar}
\email{lova.zakariasy@moov.mg}


\begin{abstract} In this paper we study the product of Toeplitz operators on the harmonic Bergman space of the unit disk of the complex plane $\mathbb{C}$. Mainly, we discuss when the product of two quasihomogeneous Toeplitz operators is also a Toeplitz operator, and when such operators commute.
\end{abstract}

\maketitle

\section{Introduction}
Let $L^2(\D,dA)$ be the space of all square integrable functions on the unit disk $\D$  with respect to the normalized Lebesgue measure $\displaystyle{dA=rdr\frac{d\theta}{\pi}}$. The harmonic Bergman space, denoted by $L^2_h$, is the closed subspace of  $L^2(\D,dA)$ consisting of all harmonic functions on $\D$. It is well known that $L^2_h$ is a Hilbert space with the set $\{\sqrt{n+1}\,z^n\}_{n=0}^\infty \cup \{\sqrt{n+1}\,\overline{z}^n\}_{n=1}^\infty$ as an orthonormal basis. Let $Q$ be the orthogonal projection of $L^2(\D,dA)$ onto $L^2_h$. For a bounded function $f$ on $\D$, the Toeplitz operator $T_f$ with symbol $f$ is defined by 
	\[T_f(u) = Q(fu), \text{ for } u \in L^2_h.\]

The function $f$ defined on $\D$ is said to be quasihomogeneous of degree $p$, if it can be written as
$f(re^{i\theta}) = e^{ip\theta}\phi(r)$
where $p$ is an integer and $\phi$ is a radial function on $\D$.  In this case, the associated Toeplitz operator $T_f$  is also called quasihomogeneous Toeplitz of degree $p$. Quasihomogeneous Toeplitz operators were first introduced by the authors while generalizing the results of \cite{cr}.

A major goal in the theory of Toeplitz operators on the Bergman space over $\mathbb{D}$  is to completely describe the commutant of a given Toeplitz operator, that is, the set of all Toeplitz operators that commute with it. Choe and Lee first in \cite{cl1} and \cite{cl2}, then recently Ding in \cite{D}, studied the commutants of Toeplitz operators with harmonic symbol, defined  on $L^2_h$. In this paper, we present new results about the commutant of a given quasihomogeneous Toeplitz operator. We  shall start by  studying the product of such operator with a radial Toeplitz operator. Then, we shall highlight the relationship between the symbols of two commuting quasihomogeneous Toeplitz operators of positive degrees.

 Before we state our results, we need to introduce the Mellin transform which
is going to be our main tool. The Mellin transform $\widehat{f}$
of a radial function $f$ in $L^1([0,1],rdr)$ is defined by
$$\widehat{f}(z)=\int_{0}^{1} f(r) r^{z-1}\,dr.$$
It is well known that, for these functions, the Mellin transform
is well defined on the right half-plane $\{z : \Re z\geq 2\}$ and
it is analytic on $\{z:\Re z>2\}$. It is important and helpful to
know that the Mellin transform $\widehat{f}$ is uniquely
determined by its values on any arithmetic sequence of integers.
In fact we have the following classical theorem \cite[p.102]{rem}.

\begin{thm}
\label{blk1}
 Suppose that $f$ is a bounded analytic function on
$\{ z : \Re z>0\}$ which vanishes at the pairwise distinct points
$z_1, z_2 \cdots$, where
\begin{itemize}
\item[i)] $\inf\{|z_n|\}>0$\\
 and
\item[ii)] $\sum_{n\ge 1}\Re(\frac{1}{z_n})=\infty$.
\end{itemize}
Then $f$ vanishes identically on $\{ z : \Re z>0\}$.
\end{thm}
\begin{remark}
\label{blk2} Now one can apply this theorem to prove that if $f\in
L^1([0,1],rdr)$ and if there exist $n_0, p\in\mathbb{N}$ such that
$$\widehat{f}(pk+n_0)=0\textrm{ for all } k\in\mathbb{N},$$
then $\widehat{f}(z)=0$ for all $z\in\{z:\Re z>2\}$ and so $f=0$.
\end{remark}

We shall often use the Lemma 2.1 in \cite[p~1767]{dz} which can be stated as follows.
\begin{lem}\label{lemdz}
Let $p \in \Z$ and $\phi$ be a bounded radial function. For each $k \in \N$,
\begin{align*}
		T_{e^{ip\theta}\phi}(z^k) = & \left\{ \begin{array}{ll}
	(2k+2p+2)\widehat{\phi}(2k+p+2)z^{k+p} & \text{ if }  k \geq -p\\
	(-2k-2p+2)\widehat{\phi}(-p+2)\overline{z}^{-k-p} & \text{ if } k < -p, 
\end{array}	\right.\\
		T_{e^{ip\theta}\phi}(\overline{z}^k) = & \left\{ \begin{array}{ll}
	(2k-2p+2)\widehat{\phi}(2k-p+2)\overline{z}^{k-p} & \text{ if }  k \geq p\\
	(2p-2k+2)\widehat{\phi}(p+2)z^{p-k} & \text{ if } k < p. 
\end{array}	\right.
\end{align*}
\end{lem}

\section{Quasihomogeneous and radial Toeplitz operators}
In \cite{lsz}, the authors gave a necessary and sufficient conditions for the product of any two quasihomogeneous Toeplitz operators, defined on the Bergman space of the unit disk, to be a Toeplitz operator. Recently in \cite{dz}, Dong and Zhou investigated the same question for Toeplitz operators defined on $L^2_h$ but with symbol of the form $e^{ip\theta}r^m$, where $p\in\mathbb{Z}$ and $m$ positive integer. In the following theorem, we prove that the product in $L^2_h$ of two Toeplitz operators, one quasihomogeneous and the other radial,  is a Toeplitz operator only in the trivial case.

\begin{thm}
Let $p$ be a nonzero integer and $\phi$ be a bounded nonzero radial function. If there exists a radial symbol $\psi$ such that $T_{e^{ip\theta}\phi}T_{\psi}$ is a Toeplitz operator, then $\psi$ must be constant function.
\end{thm}

\begin{proof}
First consider $p > 0$.
By \cite[Theorem 1.2, p~1767]{dz}, if the product $T_{e^{ip\theta}\phi}T_{\psi}$  is a Toeplitz operator, then it must be of the form 
\begin{equation}\label{h}
T_{e^{ip\theta}\phi}T_{\psi}=T_{e^{ip\theta}h},
\end{equation}
 where $h$ is a radial function. Now, using Lemma \ref{lemdz}, we show that
\begin{equation*}
T_{e^{ip\theta}\phi}T_{\psi}(z^k)=2(k+1)\widehat{\psi}(2k+2)2(k+p+1)\widehat{\phi}(2k+p+2) z^{k+p}\  \textrm{ if } k\geq 0,
\end{equation*}
and
\begin{displaymath}
T_{e^{ip\theta}\phi}T_{\psi}(\bar{z}^k)=\left\{\begin{array}{ll}
2(k+1)\widehat{\psi}(2k+2)2(k-p+1)\widehat{\phi}(2k-p+2)\bar{z}^{k-p}&\textrm{ if } k\geq p\\
2(k+1)\widehat{\psi}(2k+2)2(p-k+1)\widehat{\phi}(p+2) z^{p-k}&\textrm{ if } 0\leq k<p.
\end{array}\right.
\end{displaymath}
Similarly we have
\begin{equation*}
T_{e^{ip\theta}h}(z^k)=2(k+p+1)\widehat(2k+p+2)z^{k+p}\ \textrm{ if }k\geq 0,
\end{equation*}
and
\begin{displaymath}
T_{e^{ip\theta}h}(\bar{z}^k)=\left\{\begin{array}{ll}
2(k-p+1)\widehat{h}(2k-p+2)\bar{z}^{k-p}&\textrm{ if } k\geq p\\
2(p-k+1)\widehat{h}(p+2)z^{p-k}&\textrm{ if } 0\leq k< p.
\end{array}\right.
\end{displaymath}
Therefore, Equation (\ref{h}) together with the above equalities imply
\begin{eqnarray}
2(k+1)\widehat{\psi}(2k+2)\widehat{\phi}(2k+p+2)&=&\widehat{h}(2k+p+2),  \textrm{ if }k\geq 0\\
2(k+1)\widehat{\psi}(2k+2)\widehat{\phi}(2k-p+2)&=&\widehat{h}(2k-p+2), \textrm{ if }k\geq p\\
2(k+1)\widehat{\psi}(2k+2)\widehat{\phi}(p+2)&=&\widehat{h}(p+2),  \textrm{ if }0\leq k< p
\end{eqnarray}
Replacing $k$ by $k+p$ in (3) implies
\begin{equation}
2(k+p+1)\widehat{\psi}(2k+2p+2)\widehat{\phi}(2k+p+2)=\widehat{h}(2k+p+2), \textrm{ for all }k\geq 0.
\end{equation}
Combining (2) and (5), we obtain
\begin{equation}
\left(2(k+1)\widehat{\psi}(2k+2)-2(k+p+1)\widehat{\psi}(2k+2p+2)\right)\widehat{\phi}(2k+p+2)=0\textrm { for all } k\geq 0.
\end{equation}
Let $Z=\{k\in\mathbb{N}: \widehat{\phi}(2k+p+2)=0\}$. Since by hypothesis $\phi$ is not identically zero, Theorem \ref{blk1} implies that $\displaystyle{\sum_{k\in Z^c}\frac{1}{2k+p+2}=\infty}$, where $Z^c$ is the complementary of $Z$ in $\mathbb{N}$. Moreover, for all $k\in\Z^c$, we have
$$2(k+1)\widehat{\psi}(2k+2)=2(k+p+1)\widehat{\psi}(2k+2p+2).$$
Let  $\Un$ denotes the constant function with value one. Since $\widehat{\Un}(z)=\frac{1}{z}$, the above equation is equivalent to
\begin{equation}\label{c7}
\widehat{\Un}(2k+p+2)\widehat{\psi}(2k+2)=\widehat{\Un}(2k+2)\widehat{\psi}(2k+2p+2),\textrm{ for all }k\in Z^c.
\end{equation}
Hence, Remark \ref{blk2} together with \cite[Lemma 6, p~1468]{l}  imply that $\psi=c\Un$, for some constant $c$.

A similar argument shows that the result remains true for $p<0$, which completes the proof.
\end{proof}

In \cite{lz}, the authors showed that a quasihomogeneous Toeplitz operator, defined  on the analytic  Bergman space, commutes with a radial one, only when the radial symbol of the latter is constant. The same result remains true in $L^2_h$.
\begin{thm}\label{th3}
Let $p$ be a nonzero integer and $\phi$, $\psi$ be two bounded radial functions. If $T_{e^{ip\theta}\phi}$ commutes with $T_\psi$, then $\phi$ is zero or $\psi$ is  constant.
\end{thm}
\begin{proof}
Let $p \geq 0$.
If $T_{e^{ip\theta}\phi}T_\psi = T_\psi T_{e^{ip\theta}\phi}$, then using Lemma \ref{lemdz}, we obtain that for all $k \geq 0$
$$2(k+p+1) \widehat{\phi}(2k+p+2) \widehat{\psi}(2k+2p+2) = 2(k+1) \widehat{\phi}(2k+p+2) \widehat{\psi}(2k+2).$$
Let $Z=\{ k\geq 0: \widehat{\phi}(2k+p+2)=0\}$. If $\displaystyle{\sum_{k\in Z}\frac{1}{2k+p+2}=\infty}$, then Theorem \ref{blk1}  implies that $\phi$ is zero function.  Otherwise $\displaystyle{\sum_{k\in Z^c}\frac{1}{2k+p+2}=\infty}$, where  $Z^c$ is the complementary of $Z$  in $\mathbb{N}$. Now, for all $k\in Z^c$ we have
$$(2k+2p+2) \widehat{\psi}(2k+2p+2) = (2k+2) \widehat{\psi}(2k+2),$$
which is equivalent to
\begin{equation}\label{c8}
\widehat{\Un}(2k+2)\widehat{\psi}(2k+2p+2)=\widehat{\Un}(2k+p+2)\widehat{\psi}(2k+2),\textrm{ for all }k\in Z^c.
\end{equation}
Consequently, Remark \ref{blk2} and \cite[Lemma 6, p~1468]{l} imply that $\psi = c\Un$ for some constant $c$.

If $p < 0$, the same result is obtained by considering the adjoint operators. 
\end{proof}
Let $f$ be a  bounded function with polar decomposition 
	\begin{equation}\label{polar}
	{f(re^{i\theta})=\sum_{n\in\mathbb{Z}}e^{in\theta}f_n(r)}.
	\end{equation} 
Then we have the following corollary.
\begin{cor}\label{cor1}
If $\psi$ is a non-constant bounded radial function such that  $T_{\psi}$  commutes with $T_f$ , then $f$ must be radial function.
\end{cor}
 To prove Corollary \ref{cor1}, we need the following lemma.
\begin{lem}\label{lem1}
Let $f$ be  bounded function with polar decomposition as in \eqref{polar} and $\psi$ be a bounded radial function. Then the product $T_fT_\psi$ is commutative if and only if $T_{e^{in\theta}f_n}$ commutes with $T_\psi$, for each $n \in \Z$.
\end{lem}

\begin{proof}
Since $f$ is  bounded function, it is easy to show that the functions $f_n$ are bounded for all $n \in \Z$. For each $k > 0$ we have :
\begin{align*}
T_fT_\psi(z^k) 
& = (2k+2)\widehat{\psi}(2k+2)\sum_{n \in \Z}T_{e^{in\theta}f_n}(z^k)\\
& = (2k+2)\widehat{\psi}(2k+2)\Big(\sum_{n < -k}T_{e^{in\theta}f_n}(z^k) + \sum_{n \geq -k }T_{e^{in\theta}f_n}(z^k)\Big)\\
& = (2k+2)\widehat{\psi}(2k+2) \Big(\sum_{n < -k}2(-k-n+1)\widehat{f_n}(-n+2)\overline{z}^{-k-n} \\
&\qquad + \sum_{n \geq -k}2(k+n+1)\widehat{f_n}(2k+n+2) z^{k+n}\Big).
\end{align*}
Straight calculations imply that for  all couples of positive integers $(j,k)$ :
$$
\langle T_fT_\psi z^k , z^j \rangle = 2(2k+2)\widehat{\psi}(2k+2)\widehat{f}_{j-k}(j+k+2),$$
  and
$$\langle T_fT_\psi z^k , \overline{z}^j \rangle  = 2(2k+2)\widehat{\psi}(2k+2)\widehat{f}_{-j-k}(j+k+2).$$
Similar results are obtained when we apply $T_fT_\psi$ to $\overline{z}^k$, for all  $k > 0$. 

Henceforth, if $\displaystyle \zeta_k(z) = \left\{
\begin{array}{cl}
z^k & \text{ for } k \geq 0 \\ \overline{z}^{|k|} & \text{ for } k < 0
\end{array} \right.
$,
then 
\begin{equation*}
	\langle T_fT_\psi \zeta_k , \zeta_j \rangle = 2(2k+2)\widehat{\psi}(2k+2)C_{j,k} \text{ for all } j, k \in \Z,
\end{equation*}
where $C_{j,k}$ is one of the Mellin coefficients $\widehat{f}_{j-k}(j+k+2)$, $\widehat{f}_{-j-k}(j+k+2)$, $\widehat{f}_{k-j}(j+k+2)$ or $\widehat{f}_{k+j}(j+k+2)$.

Redoing  the same process, we have : 
\begin{equation*}
\langle T_\psi T_f \zeta_k , \zeta_j \rangle  = 2(2j+2)\widehat{\psi}(2j+2)C_{j,k} \text{ for all } j, k \in \Z.
\end{equation*}

Now assume $T_fT_\psi = T_\psi T_f$. For each $n \in \Z$ and for all $j, k > 0$ : 
\begin{align*}
\langle T_{e^{in\theta}f_n}T_\psi z^k , z^j \rangle  = & \left\{ \begin{array}{ll}
	0 & \text{ if } j \ne k+n\\
	2(2k+2)\widehat{\psi}(2k+2)\widehat{f}_{j-k}(j+k+2) & \text{ if }j = k+n 
\end{array}	\right.\\
= &  \left\{ \begin{array}{ll}
	0 & \text{ if } j \ne k+n\\
	\langle T_fT_\psi z^k , z^j \rangle & \text{ if }j = k+n
\end{array}	\right.\\
= &  \left\{ \begin{array}{ll}
	0 & \text{ if } j \ne k+n\\
	\langle T_\psi T_f z^k , z^j \rangle & \text{ if }j = k+n
\end{array}	\right.\\
= &  \left\{ \begin{array}{ll}
	0 & \text{ if } j \ne k+n\\
	2(2j+2)\widehat{\psi}(2j+2)\widehat{f}_{j-k}(j+k+2) & \text{ if }j = k+n
	\end{array}	\right.\\
= & \langle T_\psi T_{e^{in\theta}f_n}z^k , z^j \rangle.
\end{align*}
Again  similar calculations  show that for all $j, k \in \Z$,
	\[\langle T_{e^{in\theta}f_n}T_\psi \zeta_k , \zeta_j \rangle  = \langle T_\psi T_{e^{in\theta}f_n}\zeta_k , \zeta_j \rangle,
\]
so that for each $n \in \Z$, $T_{e^{in\theta}f_n}T_\psi = T_\psi T_{e^{in\theta}f_n}$.

Conversely, suppose that $T_{e^{in\theta}f_n}$ commutes with $T_\psi$. Then for all $j, k \in \Z$, there exists some $n \in \Z$ such that 
\begin{equation}
\langle T_f T_\psi \zeta_k , \zeta_j \rangle 
 =  \langle T_{e^{in\theta}f_n}T_\psi \zeta_k , \zeta_j \rangle  
 = \langle T_\psi T_{e^{in\theta}f_n}\zeta_k , \zeta_j \rangle 
 = \langle T_\psi T_f \zeta_k , \zeta_j \rangle. 
\end{equation}
Hence, $T_f T_\psi = T_\psi T_f$ and this completes the proof.
\end{proof} 

\textit{Proof of Corollary \ref{cor1}.}
Since $T_f$ commutes with $T_\psi$, Lemma \ref{lem1} implies $T_{e^{in\theta}f_n}$ commutes with $T_\psi$, for each $n\in\mathbb{Z}$.  Therefore, using Theorem \ref{th3}, we conclude that $f_n=0$ for all $n\neq 0$, and hence $f=f_0$. 
\qed

\section{Quasihomogeneous Toeplitz operators of positive degree}
In this section, we shall study the conditions under which two quasihomogeneous Toeplitz operators of positive degree commute. Direct calculations using Lemma \ref{lemdz} give the following equations.
\begin{lem}\label{lm1}  
Let $p$, $s$ be positive integers and $\phi$, $\psi$ be bounded radial functions. If $T_{e^{ip\theta}\phi}T_{e^{is\theta}\psi} = T_{e^{is\theta}\psi}T_{e^{ip\theta}\phi}$, then the following equations hold:
\begin{itemize}
\item[$\ast$] For all $k\geq 0$
\begin{equation}\label{eq1}
(k+p+1)\widehat{\phi}(2k+p+2)\widehat{\psi}(2k+2p+s+2) = (k+s+1)\widehat{\phi}(2k+2s+p+2)\widehat{\psi}(2k+s+2).
\end{equation}
\item[$\ast$] For all $ k \geq p+s$
\begin{equation}\label{eq2} 
(k-p+1)\widehat{\phi}(2k-p+2)\widehat{\psi}(2k-2p-s+2) = (k-s+1)\widehat{\phi}(2k-2s-p+2)\widehat{\psi}(2k-s+2).
\end{equation}
\item[$\ast$]  For all $\max(p,s)\leq k \leq p+s$
\begin{equation} \label{eq3}
(k-p+1)\widehat{\phi}(2k-p+2)\widehat{\psi}(s+2) = (k-s+1)\widehat{\phi}(p+2)\widehat{\psi}(2k-s+2).
\end{equation} 
\item[$\ast$] For all $0 \leq k \leq \min(p,s)$ 
\begin{equation}\label{eq6}	
 (p-k+1)\widehat{\phi}(p+2)\widehat{\psi}(2p-2k+s+2) = (s-k+1)\widehat{\phi}(2s-2k+p+2)\widehat{\psi}(s+2).
\end{equation}
\end{itemize}
Moreover,
\begin{itemize}
\item[$\ast$] If  $p \leq s$, then for all $p \leq k \leq s$
\begin{equation}\label{eq4}
(k-p+1)\widehat{\phi}(2k-p+2)\widehat{\psi}(s+2) =(s-k+1)\widehat{\phi}(2s-2k+p+2)\widehat{\psi}(s+2).
\end{equation}
\end{itemize}
\end{lem}

\begin{remark} It is important to make the distinction between the cases  "$p\leq s$"  and "$p>s$".
These two distinct conditions are crucial for the results of Theorem \ref{th6} and Theorem \ref{th7}.
\end{remark}

The following theorem shows the uniqueness of the commutant.
\begin{thm} 
Let $p$, $s$  be positive integers and let $\phi$ be a non-constant bounded radial function. If there exists a radial function $\psi$ such that $T_{e^{is\theta}\psi}$ commutes with $T_{e^{ip\theta}\phi}$, then $\psi$ is unique up to a multiplicative constant.
\end{thm}
\begin{proof}
Assume there exist two nonzero functions $\psi_1$ and $\psi_2$ such that both $T_{e^{is\theta}\psi_1}$ and $T_{e^{is\theta}\psi_2}$ commute  with $T_{e^{ip\theta}\phi}$. By Equation (\ref{eq1}),  we obtain for all $k \geq 0$ :
	\begin{align*}
(k+p+1)\widehat{\phi}(2k+p+2)\widehat{\psi_1}(2k+2p \;+&\;  s+2) = \\&(k+s+1)\widehat{\phi}(2k+2s+p+2)\widehat{\psi_1}(2k+s+2),
	\end{align*}
and 
	\begin{align*}
(k+p+1)\widehat{\phi}(2k+p+2)\widehat{\psi_2}(2k+2p \;+&\;  s+2) = \\&(k+s+1)\widehat{\phi}(2k+2s+p+2)\widehat{\psi_2}(2k+s+2).
	\end{align*}
Let $Z=\{ k\geq 0: \widehat{\phi}(2k+p+2)=0\}$. Using the same argument as in the proof of Theorem \ref{th3}, we have that $\displaystyle{\sum_{k\in Z^c}\frac{1}{2k+p+2}=\infty}$ and also that for all $k \in Z^c$ : 
\begin{equation}
\widehat{\psi_1}(2k+2p+s+2)\widehat{\psi_2}(2k+s+2) = \widehat{\psi_1}(2k+s+2)\widehat{\psi_2}(2k+2p+s+2).
\label{eq7}
\end{equation}
By Theorem \ref{blk2}, Equation (\ref{eq7}) is equivalent to 
	\[	\widehat{r^s\psi_1}(z+2p)\widehat{r^s\psi_2}(z) = \widehat{r^s\psi_1}(z)\widehat{r^s\psi_2}(z+2p) \textrm{ for }\Re z>0.
\]
Hence, \cite[Lemma 6, p~1468]{l}  implies $\psi_1 = c \psi_2$ for some constant $c$.
\end{proof}

If two quasihomogeneous symbols have the same degree, then the product of the associated Toeplitz operators is commutative only in the obvious case. 
\begin{thm}\label{th5}
Let $\phi$ and $\psi$ be bounded radial functions and $p$ be an integer. If $T_{e^{ip\theta}\phi}$ commutes with $T_{e^{ip\theta}\psi}$, then $\phi = c\psi$ where $c$ is a constant.
\end{thm}
\begin{proof}
Let $p > 0$ and assume $T_{e^{ip\theta}\phi}T_{e^{ip\theta}\psi} = T_{e^{ip\theta}\psi}T_{e^{ip\theta}\phi}$. Then Equation (\ref{eq1}) implies 
	\[\widehat{\phi}(2k+p+2)\widehat{\psi}(2k+p+2+2p) = \widehat{\phi}(2k+p+2+2p)\widehat{\psi}(2k+p+2),\ \forall k\geq 0.
\]
Now Theorem \ref{blk2} yields
	\[\widehat{\phi}(z)\widehat{\psi}(z+2p) = \widehat{\phi}(z+2p)\widehat{\psi}(z)\ \textrm{ for } \Re z>0.
\]
Therefore,  \cite[Lemma 6, p~1468]{l} provides $\phi = c\psi$.
If $p < 0$, the same result is obtained by taking the adjoint operators.
\end{proof}

Now we shall consider a Toeplitz operator with a monomial symbol. Its product with a quasihomogeneous Toeplitz operator might be either commutative or not. 
\begin{thm}\label{th6}
Let $p$, $s$ be two positive integers with $p \leq s$, and $\alpha$ be a positive real number. If there exists a radial function $\psi$ such that $T_{e^{is\theta}\psi}$ commutes with $T_{e^{ip\theta}r^\alpha}$, then either $s=p$ and $\psi = cr^\alpha$ for some constant $c$, or $\psi = 0$.
\end{thm}
\begin{proof}
First, let us  assume  that $\widehat{\psi}(s+2) \ne 0$. Since the Mellin coefficients of the monomial $\phi : r \mapsto r^\alpha$ are $\hat{\phi}(n) = \frac{1}{n+\alpha}$, for $n \geq 0$ and since  $T_{e^{ip\theta}r^\alpha}$ commutes with $T_{e^{is\theta}\psi}$,  Equation (\ref{eq4}) implies
	\[\frac{k-p+1}{2k-p+2+\alpha} = \frac{s-k+1}{2s-2k+p+2+\alpha},  \textrm{ for } p \leq k \leq s.
\]
Solving the above equality for $k$ yields to $2k = s+p$, for $k = p, p+1, \ldots, s$, which is impossible unless $p = s$. Using Theorem \ref{th5},  we conclude that $\psi = cr^\alpha$. 

Now, assume that $\widehat{\psi}(s+2) = 0$. Then Equality (\ref{eq6})  implies
\begin{equation}\label{k_0}
	\widehat{\psi}(2k+s+2) = 0, \textrm{ for } 0\leq k\leq p.
\end{equation}
Combining (\ref{eq1}) and (\ref{k_0}), we obtain
\begin{equation*}
\widehat{\psi}(2k+2p+s+2)=0, \textrm{ for all } 0\leq k\leq p,
\end{equation*}
or
\begin{equation}
\widehat{\psi}(2k+s+2)=0, \textrm{ for all } 0\leq k\leq 2p.
\end{equation}
Repeating the same argument, using each time Equation (\ref{eq1}), we show that
$$\widehat{\psi}(2k+s+2)=0, \textrm{ for all } k\geq 0.$$
Hence, Remark \ref{blk2} implies $\psi = 0$.
\end{proof}
In \cite{l}, the first author proved that for any choice of triple of positive integers $(m,p,s)$, there always exists  a radial function $\psi$ such that the Toeplitz operators $T_{e^{ip\theta}r^{(2m+1)p}}$ and $T_{e^{is\theta}\psi}$, defined on the analytic Bergman space of the unit disk, commute. We shall show that  it is not the case anymore for the analogous Toeplitz operators defined on $L^2_h$. 
\begin{thm}\label{th7}
Let $p$, $s$, $m$ be positive integers with $p > s > 0$ and $m \geq 0$, and let $\phi(r)=r^n$, where $n=(2m+1)p$.
\begin{itemize}
\item[(i)] If $p\geq m+1$ and if there exists a radial function $\psi$ such that $T_{e^{is\theta}\psi}$ commutes with $T_{e^{ip\theta}r^n}$, then $\psi$ must be the zero function.
\item[(ii)] If $p\leq m$, there exists a nonzero radial function $\psi$ such that  $T_{e^{is\theta}\psi}$ commutes with $T_{e^{ip\theta}r^n}$
\end{itemize}
\end{thm}
\begin{proof}
If $T_{e^{ip\theta}\phi}$ commutes with $T_{e^{is\theta}\psi}$, then we must have 
$$T_{e^{ip\theta}\phi}T_{e^{is\theta}\psi}(z^k)=T_{e^{is\theta}\psi}T_{e^{ip\theta}\phi}(z^k),\textrm{ for all }k\geq 0,$$
and 
$$T_{e^{ip\theta}\phi}T_{e^{is\theta}\psi}(\bar{z}^k)=T_{e^{is\theta}\psi}T_{e^{ip\theta}\phi}(\bar{z}^k),\textrm{ for all }k\geq 0.$$
Therefore, by  Lemma \ref{lm1} and since $\widehat{\phi}(z)=\frac{1}{z+n}$, we obtain the following equalities
\begin{equation}\label{eq22}
2(k+s+1)\displaystyle{\frac{\widehat{\psi}(2k+s+2)}{2k+p+2s+n+2}}=2(k+p+1)\displaystyle{\frac{\widehat{\psi}(2k+2p+s+2)}{2k+p+n+2}},\forall k\geq 0.
\end{equation}
\begin{equation}\label{eq23}
2(k-s+1)\frac{\widehat{\psi}(2k-s+2)}{2k-p-2s+n+2}=2(k-p+1)\displaystyle{\frac{\widehat{\psi}(2k-2p-s+2)}{2k-p+n+2}},\forall k\geq p+s.
\end{equation}
\begin{equation}\label{eq24}
2(k-s+1)\frac{\widehat{\psi}(2k-s+2)}{p+n+2}=2(k-p+1)\displaystyle{\frac{\widehat{\psi}(s+2)}{2k-p+n+2}},\forall p\leq k<p+s.
\end{equation}
\begin{equation}\label{eq25}
2(k-s+1)\frac{\widehat{\psi}(2k-s+2)}{p+n+2}=2(p-k+1)\displaystyle{\frac{\widehat{\psi}(2p+s-2k+2)}{p+n+2}},\forall s\leq k<p.
\end{equation}
\begin{equation}\label{eq26}
(s-k+1)\frac{\widehat{\psi}(s+2)}{p+2s-2k+n+2} =(p-k+1)\displaystyle{\frac{\widehat{\psi}(2p+s-2k+2)}{p+n+2}},\forall 0\leq k<s.
\end{equation}
Now it is easy to see that equations \eqref{eq22} and \eqref{eq23} are equivalent. In fact, by taking $j=k-p-s$ in Equation \eqref{eq23}, we obtain Equation \eqref{eq22}. We shall then use Equation \eqref{eq22} to determine the form of the radial symbol $\psi$. By setting $z=2k+2$, we complexify Equation \eqref{eq22} and we obtain
\begin{equation*}
\frac{z+2s}{z+p+2s+n}\widehat{\psi}(z+s)=\frac{z+2p}{z+p+n}\widehat{\psi}(z+2p+s) \textrm{ for } \Re z>0.
\end{equation*}
Here, we notice that the function
$$f(z)=\frac{z+2s}{z+p+2s+n}\widehat{\psi}(z+s)-\frac{z+2p}{z+p+n}\widehat{\psi}(z+2p+s)$$
is analytic and bounded in the right half-plane and vanishes for $z=2k+2$, for any $k\geq 0$. Hence, by Theorem \ref{blk1}, we have $f(z)\equiv 0$. Therefore, we obtain that in the right half-plane
\begin{equation}\label{period}
\frac{\widehat{r^s\psi}(z+2p)}{\widehat{r^s\psi}(z)}=\frac{(z+2s)(z+p+n)}{(z+p+2s+n)(z+2p)}, \textrm{ for } \Re z>0.
\end{equation}
Since $n=(2m+1)p$ and using the well-known identity $\Gamma(z+1)=z\Gamma(z)$, where $\Gamma$ is the Gamma function, we can rewrite Equation (\ref{period}) as
\begin{equation}\label{gamma}
\frac{\widehat{r^s\psi}(z+2p)}{\widehat{r^s\psi}(z)}=\frac{F(z+2p)}{F(z)}\textrm{ for }\Re z>0,
\end{equation}
where $F(z)=\displaystyle{\frac{\Gamma(\frac{z}{2p}+\frac{s}{p})\Gamma(\frac{z}{2p}+m+1)}{\Gamma(\frac{z}{2p}+\frac{s}{p}+m+1)\Gamma(\frac{z}{2p}+1)}}$. Next, Equation (\ref{gamma}), combined with \cite[Lemma 6, p~1468]{l}, implies there exists a constant C such that
\begin{equation}\label{equal}
\widehat{r^s\psi}(z)=CF(z),\textrm{ for }\Re z>0.
\end{equation}
Now, we shall show that $F(z)$ is the Mellin transform of a bounded function. Using the well-known property of the Gamma function namely $$\Gamma(z+n)=(z+n-1)(z+n-2)\ldots z\Gamma(z)\textrm{ for }n\in\mathbb{N},$$ and after simplification, we obtain that
$$F(z)=\frac{(\frac{z}{2p}+m)\ldots(\frac{z}{2p}+1)}{(\frac{z}{2p}+\frac{s}{p}+m)\ldots(\frac{z}{2p}+\frac{s}{p})},$$
which is a proper fraction in $z$ and can be written as sum of partial fractions
$$F(z)=\sum_{j=0}^{m}\frac{a_j}{z+2s+2jp}=\sum_{j=0}^{m}a_j\widehat{r^{2s+2jp}}(z).$$
Therefore, Equation (\ref{equal}) and Remark \ref{blk2} imply that
\begin{equation}\label{form}
\psi(r)=\sum_{j=0}^{m}c_jr^{s+2jp}.
\end{equation}
At this point, let us summarize what we have done so far. We proved that if there exists a radial function $\psi$ such that $T_{e^{is\theta}\psi}$ commutes with $T_{e^{ip\theta}r^n}$, then $\psi$ is given by Equation (\ref{form}). The rest of the proof will be dedicated to whether or not there exist nonzero coefficients $c_j$, $ 0\leq j\leq m$, such that $\psi$ verifies equations (14), (15) and (16). In fact,  since $\displaystyle{\widehat{\psi}(z)=\sum_{j=0}^{m}\frac{1}{z+s+2jp}}$, these three equations can be written as a homogeneous linear system in the following way
\begin{equation}\label{(S)}
\left(\begin{array}{c}
A\\
----\\
B\\
----\\
C\end{array}\right)\left(\begin{array}{c}c_0\\ \vdots\\c_m\end{array}\right)=
\left(\begin{array}{c}0\\ \vdots\\ 0 \end{array}\right),
\end{equation}
where
\begin{itemize}
\item[$\ast$] The block $A$ is of size $s\times (m+1)$ and its entries are given by:
$$a_{kj}=\frac{s-k+1}{(p+2s-2k+n+2)(s+jp+1)}-\frac{p-k+1}{(p+n+2)(p+s-k+jp+1)},$$
 for $0\leq k<s$  and  $0\leq j\leq m$.
\item[$\ast$] The block $B$ is of size $(p-s)\times(m+1)$  and its entries are given by:
$$b_{kj}=\frac{p-k+1}{(j+1)p+s-k+1}-\frac{k-s+1}{k+jp+1},$$
for $s\leq k<p$ and $0\leq j\leq m$.
\item[$\ast$] The block $C$ is of size $s\times (m+1)$ and its entries are given by:
$$c_{kj}=\frac{k-p+1}{(2k-p+n+2)(s+jp+1)}-\frac{k-s+1}{(p+n+2)(k+jp+1)},$$
for $p\leq k<p+s$ and $0\leq j\leq m$.
\end{itemize}
Before discussing the existence of solutions for our homogeneous system (\ref{(S)}), we will make three crucial observations:
\begin{itemize}
\item[$\centerdot$] $a_{s-l,j}=c_{p+l,j},$ for all $1\leq l\leq s-1$ and $0\leq j\leq m$,
 \item[$\centerdot$] $c_{pj}=-\frac{b_{sj}}{p+n+2}$, for all $0\leq j\leq m$,
 \item[$\centerdot$] The rows of the blocks $A$ and $B$ are linearly independent.
\end{itemize}
Hence we can reduce the system (\ref{(S)}) to only the blocks $A$ and $B$ and get rid of the block $C$. Moreover our system will always be  of rank $p$ (the $s$ rows of the block $A$ are linearly independent of the $p-s$ rows of the block $B$). Finally, we conclude by saying the following:
\begin{itemize}
\item[(i)] If $p\geq m+1$, the homogeneous system (\ref{(S)}) has more linearly independent equations than unknowns. In this case the only possible solution is the trivial solution i.e. $c_j=0$ for all $0\leq j\leq m$, and therefore $\psi$ is the zero function.
\item[(ii)] If $p\leq m$, the system  (\ref{(S)})  has less equations than unknowns. In this case, there exist nonzero coefficients $c_j$ which verifies (\ref{(S)}), and hence $\psi$  exists. 
\end{itemize}
\end{proof}

\textbf{Acknowledgments. } The authors would like to thank separately Elizabeth Strouse and Trieu Le for their valuable comments and useful discussions. The second author is grateful to Institut de Math\'{e}matiques de Bordeaux UMR 5251, which hosted her for a scientific stay while a part of this article was in progress.

\end{document}